\documentclass[letterpaper,11pt]{amsart}
\usepackage{graphicx, amscd, amssymb, young}
\usepackage[vcentermath]{youngtab} 

\newtheorem{lem}{Lemma}[section]
\newtheorem{thm}[lem]{Theorem}
\newtheorem{pro}[lem]{Proposition}
\newtheorem{cor}[lem]{Corollary}

\newtheorem{con}[lem]{Conjecture}

\newcommand{\ZZ}{{\mathbb{Z}}}

\newcommand{\M}{{\mathcal{M}}}

\newcommand{\U}{{\textsf{U}}}
\newcommand{\D}{{\textsf{D}}}
\renewcommand{\L}{{\textsf{L}}}

\begin{document}
\title{Skew-standard tableaux with three rows}

\author{Sen-Peng Eu}

\address{Department of Applied Mathematics, National University of Kaohsiung, Kaohsiung
811, Taiwan, ROC}
\email{speu@nuk.edu.tw}
\thanks{Partially supported by National Science Council, Taiwan under grant NSC
98-2115-M-390-002-MY3}

\maketitle

\begin{abstract}
Let $\mathcal{T}_3$ be the three-rowed strip. Recently Regev
conjectured that the number of standard Young tableaux with $n-3$
entries in the ``skew three-rowed strip'' $\mathcal{T}_3 / (2,1,0)$
is $m_{n-1}-m_{n-3}$, a difference of two Motzkin numbers. This
conjecture, together with hundreds of similar identities, were
derived automatically and proved rigorously by Zeilberger via his
powerful program and WZ method. It appears that each one is a linear
combination of Motzkin numbers with constant coefficients. In this
paper we will introduce a simple bijection between Motzkin paths and
standard Young tableaux with at most three rows. With this bijection
we answer Zeilberger's question affirmatively that there is a
uniform way to construct bijective proofs for all of those
identities.
\end{abstract}

\section{Introduction}
The enumeration of standard Young tableaux (SYTs) is a fundamental
problem in combinatorics and representation theory. For example, it
is known that the number of SYTs of a given shape $\lambda \vdash n$
is counted by the hook-length formula~\cite{Frame_54}. However, the
problem of counting SYTs of bounded height is a hard one. Let
$\mathcal{T}_k(n):=\{\lambda=(\lambda_1,\lambda_2,\dots,
\lambda_k)\vdash n : \lambda_1\ge \dots \lambda_k \ge 0\}$ be the
set of SYTs with $n$ entries and at most $k$ rows, and let
$\mathcal{T}_k=\bigcup_{n=1}^{\infty}\mathcal{T}_k(n)$ be the
$k$-rowed strip. In 1981, Regev proved that
$$|\mathcal{T}_2(n)|={n\choose \lfloor \frac{n}{2}\rfloor} \qquad \mbox{and} \qquad |\mathcal{T}_3(n)|=\sum_{i\ge 0}\frac{1}{i+1}{n\choose 2i}{2i\choose i}$$
in terms of symmetric functions~\cite{Regev_81}. Note that
$|\mathcal{T}_3(n)|$ is exactly the Motzkin number $m_n$. In 1989,
together with $|\mathcal{T}_2(n)|$ and $|\mathcal{T}_3(n)|$,
Gouyou-Beauchamps derived that
$$|\mathcal{T}_4(n)|=c_{\lfloor \frac{n+1}{2}\rfloor}c_{\lceil
\frac{n+1}{2}\rceil} \qquad \mbox{and} \qquad
|\mathcal{T}_5(n)|=6\sum_{i=0}^{\lfloor \frac{n}{2}\rfloor}{n\choose
2i}c_i\frac{(2i+2)!}{(i+2)!(i+3)!}$$ combinatorially, where
$c_n=\frac{1}{n+1}{2n\choose n}$ is the Catalan
number~\cite{Gouyou-Beauchamps_89}. His idea relied on the fact that
the number of SYTs with $n$ entries and at most $k$ rows equals the
number of involutions of $[n]$ with the length of a longest
decreasing subsequence at most $k$, hence it suffices to count these
restricted involutions. These are in fact all the simple formulae we
have for $|\mathcal{T}_k(n)|$ so far ~\cite{Stanley_99}. Meanwhile,
Zeilberger proved that for each $k$ the generating function of
$|\mathcal{T}_k(n)|$ is always $P$-recursive~\cite{Zeilberger_90}.
Gessel also pointed out this fact and derived explicitly the
exponential generating function of $|\mathcal{T}_k(n)|$ in terms of
hyperbolic Bessel functions of the first kind~\cite{Gessel_90,
Stanley_07}.

Recently Regev considered the following variation among others.
Given $\mu=(\mu_1,\mu_2, \mu_3)$ a partition of at most three parts,
let $|\mu|:=\mu_1+\mu_2+\mu_3$ and $\mathcal{T}_3(\mu; n-|\mu|)$ be
the set of SYTs with $n-|\mu|$ entries in the ``skew strip''
$\mathcal{T}_3 / \mu$. Regev conjectured that for $\mu=(2,1,0)$,
$$|\mathcal{T}_3((2,1,0); n-3)|=m_{n-1}-m_{n-3},$$
a difference of two Motzkin numbers~\cite{Regev_09}. This conjecture
is confirmed by Zeilberger by using the WZ method~\cite{EKHAD_06}.
What's more, with his powerful Maple package {\tt AMITAI},
Zeilberger could generate and rigorously prove many similar
identities, among them are a list of formulae of
$|\mathcal{T}_3(\mu;n-|\mu|)|$ for $\mu_1\le 20$, and the number of
SYTs in $\mathcal{T}_3$ with the restriction that the $(i,j)$ entry
is $m$ for $1\le m\le 15$. Amazingly, each formula is a linear
combination of negative shifts of the Motzkin numbers with constant
coefficients.

In the remark of~\cite{EKHAD_06} Zeilberger then asked that, besides
Regev's question of finding a combinatorial proof of the
$m_{n-1}-m_{n-3}$ conjecture (now a theorem, after Zeilberger), is
there a uniform way to construct combinatorial proofs to all of
these results, or prove that there is no natural bijection because
the identities are true `just because'.

In this paper we answer Regev and Zeilberger's questions
affirmatively. We shall present a simple bijection between
$\mathcal{T}_3(n)$ and the set of Motzkin paths of length $n$, which
gives another proof for $|\mathcal{T}_3(n)|=m_n$. With this
bijection we can prove Regev's conjecture and consequently all of
Zeilberger's identities for three-rowed SYTs combinatorially.

\medskip
The paper is organized as follows. We introduce the bijection in
Section 2. In Section 3 we give combinatorial proofs to Regev's and
Zeilberger's results. In the last section we give a conjecture,
regarding a relation between $|\mathcal{T}_{2\ell+1}(n)|$ and
$|\mathcal{T}_{2\ell}(n)|$.


\section{Motzkin paths and the three-rowed SYTs}
Let $m_n$ denote the $n$th Motzkin number. One way to define the
Motzkin numbers is by their generating function $M=\sum_{n\ge 0}m_n
x^n=\frac{1-x-\sqrt{1-2x-3x^2}}{2x^2}$. This function satisfies the
equation
\begin{equation}~\label{eq1}
M=1+xM+x^2M^2.
\end{equation}
One combinatorial interpretation of the Motzkin numbers is the Motzkin paths. A {\em Motzkin path} of length $n$ is a lattice path from $(0,0)$ to $(n,0)$ using {\em up steps} $(1,1)$, {\em down steps} $(1,-1)$, and {\em level steps} $(1,0)$ that never go below the $x$-axis.
Let $\U$, $\D$, and $\L$ denote an up step, a down step, and a level step, respectively.

Given a standard Young tableau $T$ with $n$ entries, we associate
$T$ with a word $\chi(T)$ of length $n$ on the alphabet
$\ZZ^{+}=\{1, 2, 3,\dots\}$, where $\chi(T)$ is obtained from $T$ by
letting the $j$th letter be the row index of the entry of $T$
containing the number $j$. The words $\chi(T)$ are known as {\em
Yamanouchi words}. For example,
\begin{equation*}
T=\young(136,257,48)
\quad \longleftrightarrow \quad \chi(T)=12132123.
\end{equation*}
On the other hand, given a Yamanouchi word $\omega$, it is
straightforward to recover the corresponding tableau
$\chi^{-1}(\omega)$, i.e., the $i$th row of which contains the
indices of the letters of $\omega$ that are equal to $i$.

\medskip
Now we present a bijection $\phi$ between Motzkin paths and the
tableaux of $\mathcal{T}_3(n)$ in terms of Yamanouchi words.

Let $\M_n$ denote the set of Motzkin paths of length $n$. Let
$\mathcal{W}_3(n) =\{\chi(T): T\in\mathcal{T}_3(n)\}$. Note that
$\mathcal{W}_3(n)$ is the set of Yamanouchi words of length $n$ on
the alphabet $\{1, 2, 3\}$. Given a $\pi\in\M_n$, let
$\pi=x_1x_2\cdots x_n$ where $x_i$ is the $i$th step of $\pi$. We
shall associate $\pi$ with a word $\phi(\pi)$ of length $n$ by the
following procedure.

\begin{enumerate}
\item[(A1)] If $\pi$ starts with a level step then we label the first step by $1$.
Otherwise $\pi$ starts with an up step. Let $j$ be the least integer such that $x_j$ is not an up step. There are two cases.
\begin{itemize}
  \item $x_j$ is a down step. Then we label the two steps $x_{j-1}$ and $x_j$ by $1$ and $2$, respectively.
  \item $x_j$ is a level step. Then we find the least integer $k$ such that $k>j$ and $x_k$ is a down step, and label the three steps $x_{j-1}$, $x_j$, and $x_k$ by $1$, $2$, and $3$, respectively.
\end{itemize}

\item[(A2)] Form a new path $\pi'$ from $\pi$ by removing those labeled steps and concatenating the remaining segments of steps. If $\pi'$ is empty then we are done, otherwise go to (A1) and proceed to process $\pi'$.
\end{enumerate}

Reading the labels of $x_1x_2\dots x_n$ in order, we obtain the
requested word $\phi(\pi)$. Note that each step with the label 2 is
preceded by a matching step with the label 1, and whenever a step is
labeled by 3 there is a matching pair of steps with labels 1 and 2.
Hence $\phi(\pi)\in \mathcal{W}_3(n)$.

For example, Figure \ref{fig:Motzkin} shows a Motzkin path $\pi$ and the corresponding Yamanouchi word $\phi(\pi)$, along with the stages of step-labeling process.
\begin{figure}[ht]
\begin{center}
\includegraphics[width=2.8in]{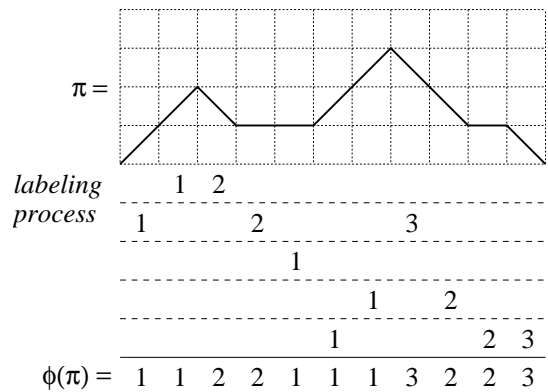}
\end{center}
\caption{\small A Motzkin path and the corresponding word.}
\label{fig:Motzkin}
\end{figure}

\noindent With $\phi(\pi)=11221113223$, we have the associated
standard Young tableau:
\[
\begin{Young}
1  & 2  & 5  & 6  & 7 \cr
3  & 4  & 9  & 10 \cr
8  & 11 \cr
\end{Young}
\]

\medskip
To find $\phi^{-1}$, given a word $\omega\in\mathcal{W}_3(n)$, let
$\omega=z_1z_2\cdots z_n$ where $z_i$ is the $i$th letter of
$\omega$. We shall recover the Motzkin path $\phi^{-1}(\omega)\in
\M_n$ by the following procedure.

\begin{enumerate}
\item[(B1)] If $\omega$ consists of letters 1 only, then each letter is associated with a level step. Otherwise, we distinguish the following two cases.
\begin{itemize}
  \item $\omega$ has no letters 3. Find the first letter 2, say $z_j$, and associate $z_{j-1}$ and $z_j$ with an up step and a down step, respectively.
  \item $\omega$ has letters 3. Suppose $C_1,\dots,C_d$ are the marks of letter $3$ in $\omega$.
       For $k=0,\dots,d-1$ and from right to left, let
       $B_{d-k}$ be the first unmarked letter 2 that
       $C_{d-k}$ encounters. For $i=1,\dots,d$, suppose
       there are $t_i$ unmarked letters 2 on the left of
       $B_i$, let $E_{i,j}$ mark the $j$th 1 from left to
       right. For $j=1,\dots,t_i$, let $D_{i,j}$ be the
       first unmarked letter 1 from right to left that
       $E_{i,j}$ encounters. We associate each pair
       $(D_{i,j},E_{i,j})$ with an up step and a down step,
       respectively. Then let $A_i$ be the first unmarked
       letter 1 on the left of $B_i$, and associate the
       triple $(A_i,B_i,C_i)$ with an up step, a level step,
       and a down step.
\end{itemize}

\item[(B2)] If all of the letters have been associated with steps then we are done, otherwise form a new word $\omega'$ by removing those letters with steps, and then go to (B1) and proceed to process $\omega'$.
\end{enumerate}
Note that at each stage of the process the path never goes below the
$x$-axis and the starting point and end point of the path are always
on the $x$-axis. Hence $\phi^{-1}(\omega)\in\M_n$. It is easy to see
that $\phi$ and $\phi^{-1}$ are indeed inverses to each other.

Now that $|\mathcal{T}_3(n)|=|\mathcal{W}_3(n)|=|\mathcal{M}_n|$, we
prove the following classic result.
\begin{thm}{\rm (Regev \cite{Regev_81})}The number of standard Young tableaux with $n$ entries and at most three rows is the Motzkin number $m_n$.
\end{thm}

\section{Regev's and Zeilberger's results}
\subsection{A Combinatorial proof to Regev's result}

Recall that $\mathcal{T}_3(\mu; n-|\mu|)$ is the set of SYTs with
$n-|\mu|$ entries in the ``skew strip'' $\mathcal{T}_3 / \mu$, where
$\mu$ is a partition of at most three parts. The following identity
is conjectured by Regev and proved by Zeilberger.

\begin{thm}{\rm (Regev \cite{Regev_09}, Zeilberger \cite{EKHAD_06})} \label{con:Regev}
$$|\mathcal{T}_3((2,1,0);n-3)|=m_{n-1}-m_{n-3}$$
\end{thm}

\medskip
We present a combinatorial proof by using the bijection $\phi$.
First we need some enumerative results. For $0\le j\le i$, let
$X(i,j;n)$ be the set of lattice paths that go from the point
$(i,j)$ to the point $(n,0)$ using steps $\U, \D,\L$ and never go
below the $x$-axis. For those paths whose starting points are on the
$x$-axis, we clearly have
\begin{equation}
|X(i,0;n)|=m_{n-i} \quad (0\le i\le n).
\end{equation}
For those paths with other starting points, we have the following
results.

\begin{pro} \label{pro:X(i,j;n)} For $0\le j\le i\le n$, the cardinality of $X(i,j;n)$ can be expressed as a linear combination of Motzkin numbers. In particular, we have
\begin{enumerate}
\item $|X(i,1;n)|=m_{n-i+1}-m_{n-i}$.
\item $|X(i,2;n)|=m_{n-i+2}-2m_{n-i+1}$.
\end{enumerate}
\end{pro}

\begin{proof} Given a $\pi\in X(i,j;n)$, the path $\pi$ can be factorized as
$\pi=\beta_1D_1\beta_2D_2\cdots \beta_jD_j\beta_{j+1}$, where $D_k$
is the first down step that goes from the line $y=j-k+1$ to the line
$y=j-k$ ($1\le k\le j$), and $\beta_k$ is a Motzkin path of certain
length (possibly empty). Hence the generating function for the
number $|X(i,j;n)|$ is $x^jM^{j+1}$, i.e.,
$|X(i,j;n)|=[x^{n-i}]\{x^jM^{j+1} \}$. With the equation
$x^2M^2=M-1-xM$, an equivalent form of (\ref{eq1}), the generating
function $x^jM^{j+1}$ can be reduced to a linear combination of
$\{x^{d}M\}_{d\in\ZZ}$. Hence $|X(i,j;n)|$ can be expressed as
linear combinations of Motzkin numbers. In particular,
\begin{equation*}
|X(i,1;n)|=[x^{n-i}]\{xM^2\}=[x^{n-i+1}]\{x^2M^2\}=[x^{n-i+1}]\{M-1-xM\}=m_{n-i+1}-m_{n-i}.
\end{equation*}
Moreover,
\begin{eqnarray*}
|X(i,2;n)| &=& [x^{n-i}]\{x^2M^3\} \\
           &=& [x^{n-i}]\{M(M-1-xM)\} \\
           &=& [x^{n-i+2}]\{M-1-xM\}-[x^{n-i}]\{M\}-[x^{n-i+1}]\{M-1-xM\} \\
           &=& m_{n-i+2}-2m_{n-i+1},
\end{eqnarray*}
as required.
\end{proof}

\noindent {\em Proof of Theorem \ref{con:Regev}.} Given a $T\in
\mathcal{T}_3((2,1,0),n-3)$, we form a new SYT $\overline{T}$ by
letting $\overline{T}(1,1)=1, \overline{T}(2,1)=2,
\overline{T}(1,2)=3$, and $\overline{T}(i,j)=T(i,j)+3$ for other
entries, where $\overline{T}(i,j)$ is the $(i,j)$ entry of
$\overline{T}$. It is clear that $\overline{T}$'s and $T$'s are
equinumerous. We are about to count the number of $\overline{T}$'s.
The associated Yamanouchi word $\omega=\chi(\overline{T})$ starts
with $1,2,1$, and according to the bijection $\phi$, the initial
three steps $x_1x_2x_3$ of the corresponding Motzkin path
$\pi=\phi^{-1}(\omega)$ of $\overline{T}$ must be one of $\{\U\D\L,
\U\D\U, \U\L\L, \U\L\U\}$ (shown below).
\begin{figure}[h]
\begin{center}
\includegraphics[width=3.2in]{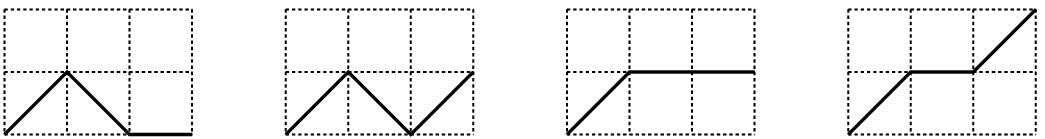}
\end{center}
\label{fig:initial-steps}
\end{figure}

Let $\pi=x_1x_2x_3\beta$, where $\beta$ is the remaining part of
$\pi$. Then $\pi$ can be classified in one of the following cases.
\begin{itemize}
   \item $x_1x_2x_3=\U\D\L$. Then $\beta$ is a Motzkin path from the point $(3,0)$ to the point $(n,0)$. The possibilities of $\pi$ are $m_{n-3}$.
   \item $x_1x_2x_3=\U\D\U$ or $\U\L\L$. Then $\beta$ starts from the point $(3,1)$. By Proposition \ref{pro:X(i,j;n)}(i), the possibilities of $\pi$ are $2(m_{n-2}-m_{n-3})$.
   \item $x_1x_2x_3=\U\L\U$. Then $\beta$ starts from the point $(3,2)$. By Proposition \ref{pro:X(i,j;n)}(ii), the possibilities of $\pi$ are $m_{n-1}-2m_{n-2}$.
\end{itemize}
The assertion follows from summing up the above three quantities.
\qed

\subsection{A uniform way to prove Zeilberger's identities}
Note that the method we used to prove Regev's conjecture can be
applied to tableaux of various skew shapes. Let $\mu$ be a partition
with at most three parts. The purpose is to compute
$|\mathcal{T}_3(\mu; n-|\mu|)|$ for a fixed SYT $T'$ of shape $\mu$.
For every $T\in \mathcal{T}_3(\mu; n-|\mu|)$, let $\overline{T}$ be
the SYT defined by
$$\overline{T}(i,j)=\begin{cases}
T'(i,j), &\text{if the $(i,j)$ entry is in $T'$;}\\
T(i,j)+|\mu|, &\text{if the $(i,j)$ entry is in $T$.}
\end{cases}$$ Converting $\overline{T}$ into a Yamanouchi word $\omega
=\chi(\overline{T})$, it suffices to count the Motzkin paths
$\pi=\phi^{-1}(\omega)$. Such paths $\pi$ have a factorization
$\pi=\alpha \beta$, where $\alpha$ consists of the initial $|\mu|$
steps and $\beta$ is the remaining part of $\pi$. Now one can list
all possible segments $\alpha$ whose words coincide with the initial
subword of length $|\mu|$ of $\chi(\overline{T})$, and can classify
these segments according to their end points, say $(i,j)$, in the
plane. Then for each class the possibilities of $\beta$ can be
determined by the formulae $|X(i,j;n)|$ in Lemma~\ref{pro:X(i,j;n)}.
Since each term $|X(i,j;n)|$ can be expressed as a linear
combination of Motzkin numbers, we answer Zeilberger's question.

\begin{thm}
For every $\mu$ of at most three parts, the cardinality of
$\mathcal{T}_3(\mu; n-|\mu|)$ can be expressed as a linear
combination of negative shifts of the Motzkin numbers $m_n$ with
constant coefficients, and each formula can be proved
combinatorially.
\end{thm}

Zeilberger also considered the problem of counting those SYTs in
$\mathcal{T}_3$ with the restriction that the $(i,j)$ entry is a
fixed number $m$. He pointed out that the formula can be expressed
as a linear combination of $|\mathcal{T}_3(\mu;n-|\mu|)|$ (hence
also a linear combination of Motzkin numbers), and produced a list
of formulae for $1\le m\le 15$ and all feasible
$(i,j)$~\cite{EKHAD_06}. Therefore, by using the same method we can
also prove all these formulae combinatorially.

\section{A conjecture}
Although obtaining simple formulae for $|\mathcal{T}_k(n)|$, $k\ge
6$, seems hopeless, in this section we give a conjecture which
reveals an (unexpected) relation between $|\mathcal{T}_{2\ell}(n)|$
and $|\mathcal{T}_{2\ell+1}(n)|$. The proof of the following simple
fact is omitted.
\begin{lem} The number of
Motzkin paths of length $n$ with the restriction that the level
steps $(1,0)$ are always on the $x$-axis is the central binomial
number ${n\choose \lfloor \frac{n}{2} \rfloor}$.
\end{lem}
Let $\mathbb{R}_{\ge 0}$ denote the set of nonnegative real numbers.
Hence we have the following fact:
\begin{cor}
$|\mathcal{T}_3(n)|$ equals the number of lattice paths in
$\mathbb{R}_{\ge 0}^2$ from the origin to the $x$-axis using steps
$(1,0), (1,1), (1,-1)$, and $|\mathcal{T}_2(n)|$ equals the number
of these lattice paths with the restriction that the $(1,0)$ steps
appear only on the $x$-axis.
\end{cor}

This leads to our conjecture. Let $\{\mathbf{e}_1,\dots
\mathbf{e}_{\ell+1}\}$ denote the standard basis of
$\mathbb{R}^{\ell+1}$ and let $\mathcal{L}_{2\ell+1}(n)$ be the set
of $n$-step lattice paths in $\mathbb{R}_{\ge 0}^{\ell+1}$ from the
origin to the axis along $\mathbf{e}_1$, using $2\ell+1$ kinds of
steps $\mathbf{e}_1, \mathbf{e}_1\pm \mathbf{e}_2, \mathbf{e}_1\pm
(\mathbf{e}_2-\mathbf{e}_3), \mathbf{e}_1\pm
(\mathbf{e}_3-\mathbf{e}_4), \dots , \mathbf{e}_1\pm
(\mathbf{e}_{\ell}-\mathbf{e}_{\ell+1})$. By combining works of
Grabiner and Magyar~\cite{Grabiner_93} and Gessel~\cite{Gessel_90},
Zeilberger proved~\cite{Zeilberger_07}, equivalently, that
$$|\mathcal{T}_{2\ell+1}(n)|=|\mathcal{L}_{2\ell+1}(n)|.$$

\begin{con} Let $\mathcal{L}_{2\ell}(n)$ be the set of lattice
paths in $\mathcal{L}_{2\ell+1}(n)$ with the restriction that the
$\mathbf{e}_1$ steps appear only on the hyperplane spanned by
$\{\mathbf{e}_1, \dots , \mathbf{e}_\ell\}$. Then we have
$$\qquad|\mathcal{T}_{2\ell}(n)|=|\mathcal{L}_{2\ell}(n)|.$$
\end{con}

This conjecture has been proved for $\ell=2,3$ and checked by
computer for $\ell \le 10$ up to $n=30$~\cite{Eu_10}.

\section*{Acknowledgements} The author thanks T.-S. Fu and D. Zeilberger for helpful discussions,
and the referee for the careful reading and valuable suggestions.


\vspace{1cm} \baselineskip=16pt
\begin{thebibliography}{99}

\bibitem{EKHAD_06} S. B. Ekhad, D. Zeilberger, Proof of a conjecture of Amitai Regev about three-rowed Young tableaux (and much more!),
The Personal Journal of Shalosh B. Ekhad and Doron Zeilberger
(2006), \verb+www.math.rutgers.edu/~zeilberg/pj.html+

\bibitem{Eu_10} T.-Y. Cheng, S.-P. Eu, J. Hou, T.-W. Hsu, Higher Motzkin paths, in preparation.

\bibitem{Frame_54} J. S. Frame, G. de B. Robinson, R. M. Thrall, The hook graphs of the symmetric group, Canad. J. Math. 6 (1954) 316--325.

\bibitem{Gessel_90} I. Gessel, Symmetric functions and P-recursiveness, J. Combin. Theory Ser. A, 53 (1990) 257--285.

\bibitem{Gouyou-Beauchamps_89} D. Gouyou-Beauchamps, Standard Young tableaux of height 4 and 5, European J. Combin. 10 (1989) 69--82.

\bibitem{Grabiner_93} D. J. Grabiner, P. Magyar, Random walks in Weyl chambers and the
decomposition of tensor powers. J. Algebraic Combin. 2 (1993)
239--260.

\bibitem{Regev_81} A. Regev, Asymptotic values for degrees associated with strips of Young diagrams, Adv. Math. 41 (1981) 115--136.

\bibitem{Regev_09} A. Regev, Probabilities in the $(k, l)$ hook, Israel J. Math. 169 (2009) 61--88.

\bibitem{Stanley_99} R. Stanley, Enumerative Combinatorics, Volume 2, Cambridge Studies in Advanced Mathematics 62, Cambridge University Press, Cambridge, 1999.

\bibitem{Stanley_07} R. Stanley, Increasing and decreasing subsequences and their variants, Proc. Internat. Cong. Math. (Madrid, 2006), vol. 1,
American Mathematical Society, Providence, RI, 2007, pp. 545--579.

\bibitem{Zeilberger_90} D. Zeilberger, A holonomic systems approach to special
functions identities, J. Comput. Appl. Math. 32 (1990) 321--368.

\bibitem{Zeilberger_07} D. Zeilberger, The number of ways of walking in $x_1\ge \dots \ge x_k\ge 0$ for $n$ days,
starting and ending at the origin, where at each day you may either
stay in place or move one unit in any direction, equals the number
of $n$-cell standard Young tableaux with $\le 2k+1$ rows, The
Personal Journal of Shalosh B. Ekhad and Doron Zeilberger (2007),
\verb+www.math.rutgers.edu/~zeilberg/pj.html+

\end{thebibliography}
\end{document}